\documentclass[10pt, leqno]{article}

\usepackage{mathrsfs}
\usepackage{amsmath}
\usepackage{amssymb}
\usepackage{amsthm}
\usepackage{comment}

\usepackage{accents}

\newtheorem{definition}{Definition}[section]

\newtheorem{prop}{Proposition}[section]
\newtheorem{lem}{Lemma}[section]
\newtheorem{cor}{Corollary}[section]

\numberwithin{equation}{section}

\def\ep{\epsilon}

\begin{document}

\title{Notes on
symmetrization by B\'ezoutiant\\
\vspace{2mm}\normalsize{ In memoria di Enrico Jannelli}}

\author{Tatsuo Nishitani\footnote{Department of Mathematics, Osaka University, Machikaneyama 1-1, Toyonaka 560-0043, Japan:  
nishitani@math.sci.osaka-u.ac.jp
}}

\date{}
\maketitle

\def\dif{\partial}
\def\al{\alpha}
\def\be{\beta}
\def\ga{\gamma}
\def\om{\omega}
\def\lam{\lambda}
\def\tika{{\tilde \kappa}}
\def\baka{{\bar \kappa}}
\def\varep{\varepsilon}
\def\tal{{\tilde\alpha}}
\def\tbe{{\tilde\beta}}
\def\tis{{\tilde s}}
\def\bas{{\bar s}}
\def\R{{\mathbb R}}
\def\N{{\mathbb N}}
\def\C{{\mathbb C}}
\def\Q{{\mathbb Q}}
\def\Ga{\Gamma}
\def\La{\Lambda}
\def\lr#1{\langle{#1}\rangle}
\def\mD{\lr{ D}_{\mu}}
\def\xim{\lr{\xi}_{\mu}}

\begin{abstract}Let $p$ and $q$ be a monic hyperbolic polynomials such that $q$ separates $p$  and let $H$ be the B\'ezoutian (form) of $p$ and $q$. Then $H$ is nonnegative definite and symmetrizes the Sylvester matrix associated with $p$. We give a simple proof of this fact and at the same time prove that the family of B\'ezoutian  of Nuij approximation of $p$ and $p'$  gives  quasi-symmetrizers introduced by S.Spagnolo. A relation connecting $H$ with the symmetrizer which was used by J.Leray for strictly hyperbolic polynomials is also given.
\end{abstract}

 
 
 \section{B\'ezoutian as energy form}
 
We recall well known facts about polynomials whose roots separate the roots of other polynomials and energy forms obtained using these polynomials.  Let $p(\zeta)=\zeta^m+a_1\zeta^{m-1}+\cdots+a_m$ be a monic polynomial of degree $m$ and consider the  differential operator
\[
p(D_t)u=\sum_{j=0}^ma_{m-j}D^j_tu,\qquad D_t=\frac{1}{i}\frac{d}{dt}
\]
where $p(\zeta)$ is assumed to be hyperbolic, that is all the roots are real so that one can write $
p(\zeta)=\prod_{j=1}^m(\zeta-\lambda_k)$ with  $\lambda_k\in\R$.
For a polynomial  in $(\zeta, {\bar\zeta})$ 
\[
 h(\zeta,\bar\zeta)=\sum_{i,j=0}^{m-1}h_{ij}\zeta^i{\bar\zeta}^{j},\quad h_{ij}\in \C
\]
 we associate a differential quadratic form
$$
\hat h(Du,\overline{Du})=\sum_{i,j=0}^{m-1}h_{ij}D^i_tu\cdot\overline{D^j_tu}
$$
where $Du=(u,D_tu,\ldots,D_t^{m-1}u)$. It is easy to see that
$$
D_t \hat h(Du,\overline{Du})=\sum_{i,j=0}^{m-1}h_{ij}(D_t^{i+1}u\cdot\overline{D_t^ju}-
D_t^{i}u\cdot\overline{D_t^{j+1}u})=\hat g(Du,\overline{Du})
$$
where $\hat g(z,\bar z)$ is the quadratic form associated to $g(\zeta,{\bar\zeta})=(\zeta-{\bar\zeta})h(\zeta,{\bar\zeta})$. For a polynomial $p(\zeta)$
associate a linear form in $z=(z_0,z_1,\ldots,z_m)\in\C^{m+1}$ by
$$
\hat p(z)=\sum_{j=0}^ma_jz_j
$$
such that $\hat p(Du)=p(D_t)u$. It is clear that
\begin{lem} 
\label{lem:threeone} If $h(\zeta,{\bar\zeta})=p(\zeta)q({\bar\zeta})$ with real polynomials $p$ and $q$ then one has $
\hat h(z,\bar z)=\hat p(z)\hat q(\bar z)$. 
\end{lem}
To relate $\hat g(Du,\overline{Du})$ to $p(D_t)$ 
assume 
\begin{equation}
g(\zeta,{\bar\zeta})=p(\zeta)q({\bar\zeta})+p({\bar\zeta})r(\zeta)
\label{eq:eqthreetwo}
\end{equation}
with some {\it real} polynomials $q$ and $r$. Then by Lemma \ref{lem:threeone} we obtain
\begin{equation}
\label{eq:dif:h}
\frac{d}{dt}\hat h(Du,\overline{Du})=i\big( p(D_t)u\cdot\overline{q(D_t)u}+\overline{p(D_t)u}
\cdot r(D_t)u \big).
\end{equation}
From \eqref{eq:eqthreetwo} it follows $
(\zeta-{\bar\zeta})h(\zeta,{\bar\zeta})=p(\zeta)q({\bar\zeta})+p({\bar\zeta})r(\zeta)$. 
Taking $\zeta={\bar\zeta}$ one sees that $r(\zeta)=-q(\zeta)$ and hence
$$
h(\zeta,{\bar\zeta})=\frac{p(\zeta)q({\bar\zeta})-p({\bar\zeta})q(\zeta)}{\zeta-{\bar\zeta}}
$$
which is called the B\'ezout form of $p$ and $q$, or B\'ezoutian of $p$ and $q$ (see for example \cite{RS}).
Denote by $p_k(\zeta)$ the monic polynomial of degree $m-1$ of which roots  are $\lambda_j$, $1\leq j\leq m$, $j\neq k$, that is
\begin{equation}
\label{eq:psita:k}
p_k(\zeta)=\prod_{j\neq k}^m(\zeta-\lambda_j).
\end{equation}
A hyperbolic polynomial $p$ is called {\it strictly hyperbolic} if all the roots $\lambda_k$ are different from each other. 
\begin{lem} 
\label{lem:threetwo} Assume that $p$ is a strictly hyperbolic polynomial and $q$ is a real 
polynomial of degree at most $m-1$. If there 
exists $c>0$ such that
\begin{equation}
{\hat h}_{p,q}(z,{\bar z})\geq c\sum_{k=1}^m\big|{\hat p}_k(z)\big|^2
\label{eq:threeseven}
\end{equation}
then $q(\zeta)$ is a hyperbolic polynomial  with positive coefficient of 
$\zeta^{m-1}$ and separates $p(\zeta)$, that is the zeros $\{\mu_k\}$ 
of $q(\zeta)$ verify
\begin{equation}
\label{eq:def:bunri}
\lambda_1<\mu_1<\lambda_2<\cdots <\lambda_{m-1}<\mu_{m-1}<\lambda_m.
\end{equation}
Conversely if $q(\zeta)$ is a hyperbolic polynomial of degree $m-1$ with positive coefficients of $\zeta^{m-1}$ which separates $p$ then \eqref{eq:threeseven} holds with some $c>0$.
\end{lem}
\begin{cor} 
\label{cor:threethree} Assume that $p(\zeta)$ is a strictly hyperbolic polynomial
 and $q(\zeta)$ is a hyperbolic polynomial of degree $m-1$ with positive coefficient of $\zeta^{m-1}$ which 
separates $p(\zeta)$. Let $r(\zeta)$ 
be a polynomial of degree $m-1$. Then there is a $C>0$ such that
$$
C\hat h_{p,q}(z,\bar z)\geq |\hat r(z)|^2,\quad z\in \C^m.
$$
\end{cor}
The B\'ezoutiant of $p$ and $p'$ provides a positive definite energy form and plays a fundamental role in studying strictly hyperbolic equations in \cite{Ga} (see also  \cite{Leray}, \cite{Sa}).

Next we study general monic hyperbolic polynomials. Let 
$$
p(\zeta)=\prod_{j=1}^s(\zeta-\lambda_{(j)})^{r_j},\quad \lambda_{(j)}\in\R,\quad 
\sum_{j=1}^sr_j=m
$$
where $\lambda_{(j)}$ are different from each other. We also write the same $p(\zeta)$ 
as $
p(\zeta)=\prod_{j=1}^m(\zeta-\lambda_j)$
so that $\{\lambda_1,\ldots,\lambda_m\}=\{\lambda_{(1)},\ldots,\lambda_{(1)},\lambda_{(2)},\ldots,\lambda_{(2)},\ldots\}$ and $p_k(\zeta)$ is still defined by \eqref{eq:psita:k}.
Condition \eqref{eq:def:bunri}  could be generalized  to
\begin{definition} 
\rm
\label{dnfourone} Let $q(\zeta)$ be a hyperbolic polynomial of degree $m-1$. Then 
$q(\zeta)$ separates $p(\zeta)$ if $q$ has the form
$$
q(\zeta)=c\prod_{j=1}^s(\zeta-\lambda_{(j)})^{r_j-1}\prod_{j=1}^{s-1}(\zeta-\mu_j),\quad
 (c\neq 0)
$$
where $\lambda_{(1)}<\mu_1<\lambda_{(2)}<\cdots<\mu_{s-1}<\lambda_{(s)}$.
\end{definition}
\begin{lem} 
\label{lem:fourone} Assume that $q(\zeta)$ is a real polynomial of degree $m-1$ with positive coefficient of 
$\zeta^{m-1}$ and separates $p(\zeta)$. Then there exists $c>0$ such that
\begin{equation}
\label{eq:bunri}
{\hat h}_{p,q}(z,{\bar z})\geq c\,\sum_{k=1}^m|{\hat p}_k(z)|^2,\quad z\in \C^m.
\end{equation}
Conversely if \eqref{eq:bunri} is satisfied for a real polynomial $q$ of degree $m-1$ then $q$ is hyperbolic with positive coefficients of $\zeta^{m-1}$ and separates $p$.
\end{lem}
\begin{lem} 
\label{lem:fourfour} Let $p'=\dif p/\dif \zeta$. Then there exists $c>0$ such that
$$
{\hat h}_{p,p'}(z,\bar z)=\sum_{k=1}^m|{\hat p}_k(z)|^2\geq c\,|{\hat p'}(z)|^2.
$$
\end{lem}
\begin{cor} 
\label{cor:fourthree} Let $p(\zeta)$ be a hyperbolic polynomial. Then $p'(\zeta)$ 
separates $p(\zeta)$.
\end{cor}
For the sake of completeness we give proofs of  Lemmas \ref{lem:fourone} and  \ref{lem:fourfour} in the last section.

Denoting $p^{(j)}=\dif^j p/\dif \zeta^j$ it follows from \eqref{eq:dif:h} that
\begin{equation}
\label{eq:hD:dif}
\begin{split}
\frac{d}{dt}{\hat h}_{p^{(j)},p^{(j+1)}}(Du,\overline{Du})=-2{\mathsf{Im}}\,\big(p^{(j)}(D_t)u,\overline{p^{(j+1)}(D_t)u}\big)\\
\leq 2\,\big|p^{(j)}(D_t)u\big|\big|p^{(j+1)}(D_t)u\big|.
\end{split}
\end{equation}
Since $p^{(j+1)}$ separates $p^{(j)}$ in view of Corollary \ref{cor:fourthree}, from Lemma  \ref{lem:fourfour} one has $c_j\big|p^{(j+1)}(D_t)u\big|\leq {\hat h}_{p^{(j)},p^{(j+1)}}(Du,\overline{Du})$ with $c_j>0$. 
Assume that the multiplicity of the roots of $p(\zeta)$ is at most $r$ then $p^{(r-1)}$ is a strictly hyperbolic polynomial and hence $|u(t)|^2\leq C_r{\hat h}_{p^{(r-1)},p^{(r)}}(Du,\overline{Du})$ by Corollary \ref{cor:threethree}. Thus
\[
\gamma^{2r}\int_{-\infty}^te^{-2\gamma s}|u(s)|^2ds\leq C\int_{-\infty}^te^{-2\gamma s}\big|p(D_t)u\big|^2ds.
\]
A family of  energy forms defined by B\'ezoutiant of $p$ and $q$, where $q$ is taken not only to be $p'$ but also  perturbations of $p'$,  plays an  important role in studying even weakly hyperbolic equations in \cite{Iv}, \cite{N1}.

\section{Symmetrization by B\'ezoutiant}

With $U={^t}(u,D_tu,\ldots,D_t^{m-1}u)$ the equation $p(D_t)u=0$ is reduced to $D_tU=AU$ where $A$ is called the Sylvester matrix associated to 
$p(\zeta)$;
\[
A=\left(\begin{array}{ccccc}
0&1&\cdots&0\\
\vdots&\vdots&\ddots\\
0&0&\cdots&1\\
-a_m&-a_{m-1}&\cdots&-a_1
\end{array}\right).
\]
Let $h(\zeta,\bar\zeta)=\sum_{i,j=0}^{m-1}h_{ij}\zeta^i{\bar\zeta}^j$ be the B\'ezoutian of $p$ and $q$ then we call conveniently  the matrix $H=(h_{ij})$  the B\'ezout matrix of $p$ and $q$. 
\begin{prop}
\label{pro:hyp:sym}
Assume that $p$ is a monic hyperbolic polynomial and $q$ separates $p$. Let $H$ be the B\'ezout matrix of $p$ and $q$. Then $H$ is nonnegative definite and symmetrizes $A$, 
that is $HA$ is symmetric and ${\rm det }\,H$  
is the resultant of $p$ and $q$. 
\end{prop}
\begin{proof} The fact that $H$ is nonnegative definite is obvious from Lemma \ref{lem:fourone} because
\begin{equation}
\label{eq:H:pk}
(Hz,z)={\hat h}_{p,q}(z,\bar z)\geq c\,\sum_{k=1}^m\big|{\hat p}_k(z)\big|^2.
\end{equation}
Next show that $H$ symmetrizes $A$. First treat the case when $p(\zeta)$ is a strictly hyperbolic polynomial. By the Lagrange interpolation formula, one can write
\begin{equation}
\label{eq:hokan}
q(\zeta)=\sum_{k=1}^m\frac{q(\lambda_k)\,p(\zeta)}{p'(\lambda_k)\,(\zeta-\lambda_k)}=\sum_{k=1}^m\al_k\,p_k(\zeta),\quad \al_k=\frac{q(\lambda_k)}{p'(\lambda_k)}.
\end{equation}
Then one has
\begin{eqnarray*}
p(\zeta)q({\bar\zeta})-p({\bar\zeta})q(\zeta)=\sum_{k=1}^m\al_k\,p_k(\zeta)p_k({\bar\zeta})(\zeta-\lambda_k)
-\sum_{k=1}^m\al_k\,p_k({\bar\zeta})p_k(\zeta)({\bar\zeta}-\lambda_k)\\
=(\zeta-{\bar\zeta})\sum_{k=1}^m\al_k\,p_k(\zeta)p_k({\bar\zeta})
=(\zeta-{\bar\zeta})\sum_{k=1}^m\al_k\,|p_k(\zeta)|^2
\end{eqnarray*}
which gives $
h_{p,q}(\zeta,{\bar\zeta})=\sum_{k=1}^m\al_k|p_k(\zeta)|^2$ and hence
\begin{equation}
\label{eq:Lagra}
{\hat h}_{p,q}(z,{\bar z})=\sum_{k=1}^m\al_k|{\hat p}_k(z)|^2
\end{equation}
where $\alpha_k>0$ by \eqref{eq:H:pk}.

Denote the  elementary symmetric polynomials in $(\lambda_1,\ldots,\lambda_{k-1},\lambda_{k+1},\ldots,\lambda_m)$ by
\[
\sigma_{\ell,k}=\sum_{1\leq j_1<\cdots<j_{\ell}\leq m, j_p\neq k}\lambda_{j_1}\cdots\lambda_{j_{\ell}},\quad \sigma_{0,k}=1,\;\; \ell=0,1,\ldots,m-1.
\]
Since ${\hat p}_k(z)=\sum_{i=1}^m(-1)^{m-i}\sigma_{m-i,k}z_{i-1}$ it follows from \eqref{eq:H:pk} and \eqref{eq:Lagra} that
 \begin{equation}
 \label{eq:hij}
 h_{ij}=\sum_{k=1}^m(-1)^{i+j}\alpha_k\,\sigma_{m-i,k}\sigma_{m-j,k}.
 \end{equation}
Denoting by $R$ the Vandermonde's matrix;
 \[
 R=\left(\begin{array}{ccccc}
1&1&\cdots&1\\
\lambda_1&\lambda_2 &\cdots&\lambda_m\\
\vdots&\vdots&\ddots&\vdots\\
\lambda_1^{m-1}&\lambda_2^{m-1}&\cdots&\lambda_m^{m-1}
\end{array}\right)
\]
it is clear that
\begin{equation}
\label{eq:AR=RD}
AR=R\left(\begin{array}{ccc}
\lambda_1&\\
&\ddots&\\
&&\lambda_m
\end{array}\right).
\end{equation}
Denote by $^{co}\!R=(r_{ij})$ the cofactor matrix of $R$ and by $\Delta(
\lambda_1,\ldots,\lambda_k)$ the difference-product of $\lambda_1,\ldots,\lambda_k$.
 It is easily seen that $r_{ij}$ is divisible by $\Delta_i=\Delta(\lambda_1,\ldots,\lambda_{i-1},\lambda_{i+1},\ldots,\lambda_m)$, hence
 \begin{equation}
 \label{eq:2.1}
 r_{ij}=g_{ij}(\lambda_1,\ldots,\lambda_{i-1},\lambda_{i+1},\ldots,\lambda_m)\Delta_i.
 \end{equation}
Since $r_{ij}$ and $\Delta_i$ are alternating polynomials in $(\lambda_1,\ldots,\lambda_{i-1},\lambda_{i+1},\ldots,\lambda_m)$ of degree $m(m-1)/2-j+1$ and  $(m-1)(m-2)/2$ respectively, as a result $g_{ij}$ is a symmetric polynomial in $(\lambda_1,\ldots,\lambda_{i-1},\lambda_{i+1},\ldots,\lambda_m)$ of degree
\[
 m-j=m(m-1)/2-j+1-(m-1)(m-2)/2.
 \]
  Therefore   it follows that $g_{ij}$ is a polynomial in $\sigma_{\ell,i}$. Noting that $\Delta_i$ is of degree $m-2$ and $r_{ij}$ ($j\neq m$) is of degree $m-1$ respectively with respect to $\lambda_{\ell}$ ($\ell\neq i$),  one concludes that $g_{ij}$ is of degree $1$ with respect to $\lambda_{\ell}$ ($\ell\neq i$) which proves that
\begin{equation}
\label{eq:2.2}
g_{ij}=(-1)^{i+j}\sigma_{m-j,i}.
\end{equation}
Thus denoting $G=(g_{ij})$ it follows from \eqref{eq:hij} and \eqref{eq:2.2} that
\[
H=(h_{ij})={^tG}\Lambda G, \quad \Lambda={\rm diag}(\alpha_1,\ldots,\alpha_m)
\]
here another proof of  the nonnegative definiteness of the B\'ezout matrix $H$ of $p$ and $q$.

Set $D={\rm diag}\,(\Delta_1,\ldots,\Delta_m)$ and note that $D$ is invertible. It follows from \eqref{eq:2.1} that $G=D^{-1}(^{co}\!R)=({\rm det}\,R)D^{-1}R^{-1}$ and hence 
\[
GAG^{-1}=D^{-1}(R^{-1}AR)D.
\]
It is clear that $\Lambda GAG^{-1}$ is a diagonal matrix because $R^{-1}AR$, $D$ and $\Lambda$ are diagonal matrices. Then $\Lambda GAG^{-1}={^t}G^{-1}\,{^t}\!A{^t}G\Lambda$ yields ${^t}G\Lambda GA={^t}\!A{^t}G\Lambda G$ which proves that $HA$ is symmetric. From $G=({\rm det}\,R)D^{-1}R^{-1}$ it follows that
\begin{equation}
\label{eq:Cform}
G={\rm diag}\,\big(\pm p_{1}(\lambda_1),
\pm\, p_{2}(\lambda_2),\ldots, \pm\, p_m(\lambda_m)\big)R^{-1}
\end{equation}
and hence $\big({\rm det}\,G\big)^2=\prod_{j=1}^mp_j(\lambda_j)^2/\Delta^2$ where $\Delta={\rm det}\,R=\Delta(\lambda_1,\ldots,\lambda_m)$. Consequently, since $p_j(\lambda_j)=p'(\lambda_j)$ one has
\[
{\rm det}\,H
=\frac{1}{\Delta^2}\prod_{j=1}^mp_j(\lambda_j)^2\prod_{j=1}^m\alpha_j=\frac{1}{\Delta^2}\prod_{j=1}^mp_j(\lambda_j)\prod_{j=1}^mq(\lambda_j)=\prod_{j=1}^mq(\lambda_j)
\]
which is the resultant of $p$ and $q$ and this completes the proof for strictly hyperbolic polynomial $p(\zeta)$.

Passing to the general case, following  \cite{Nuij} introduce Nuij approximation of $p$ and $q$
\begin{equation}
\label{eq:nuij}
p_{\ep}(\zeta) =\big(1 + \ep\,(d/d\zeta)\big)^{m-1} p(\zeta),\quad q_{\ep}(\zeta) =\big(1 + \ep\,(d/d\zeta) \big)^{m-1} q(\zeta)
\end{equation}
for $\ep>0$. Making a closer look at the Nuij approximation one has
\begin{lem}
\label{lem:pq:kinji}
For $\ep> 0$, both $p_{\ep}$ and $q_{\ep}$ are strictly hyperbolic and $q_{\ep}$ separates $p_{\ep}$. Write $p_{\ep}(\zeta)=\prod_{j=1}^m\big(\zeta-\lambda_j(\ep)\big)$ where $\lambda_1(\ep)\leq \lambda_2(\ep)\leq \cdots\leq \lambda_m(\ep)$ then one can find  $c>0$ depending only on $m$ such that
\begin{equation}
\label{eq:nuij:sa}
\lambda_{k+1}(\ep)-\lambda_k(\ep)\geq c \,\ep,\quad j=1,\ldots,m-1.
\end{equation}
\end{lem}
For the sake of completeness we give a proof in the last section.
 
 Let $A_{\ep}$ be the Sylvester matrix associated with $p_{\ep}$ and let $H_{\ep}={^tG}_{\ep} \Lambda_{\ep}G_{\ep}$ be the B\'ezout matrix of  $p_{\ep}$ and $q_{\ep}$. Note that every entry of $H_{\ep}$  is a polynomial in coefficients of $p_{\ep}$ and $q_{\ep}$ by definition, hence obviously, as $\ep \to 0$, we have $A_{\ep} \rightarrow A$, $H_{\ep}\to H$, for the coefficients of $p_{\ep}(\zeta)$ and $q_{\ep}(\zeta)$  go to the ones of $p(\zeta)$ and $q(\zeta)$. Similarly the resultant of $p_{\ep}$ and $q_{\ep}$ converges to that of $p$ and $q$.
Letting $\ep \to 0$ we obtain the result.
\end{proof}  
The next corollary is found in \cite{Ja2}, \cite{Nbook}.
\begin{cor}
\label{cor:Asakura} Assume that $p$ is a monic hyperbolic polynomial and let $H$ be the B\'ezout matrix of $p$ and $p'=\dif p/\dif \zeta$. Then $H$ is nonnegative definite and symmetrizes $A$ and ${\rm det }\,H$  
is  the discriminant of $p$.
\end{cor}
%

\section{Quasi-symmetrizers by B\'ezout matrices}

Let $p(\zeta)$ be a monic hyperbolic polynomial of degree $m$. Assume that one can find a family of monic strictly hyperbolic polynomials $\{p_{\ep}(\zeta)\}_{\ep>0}$ of degree $m$
\[
p_{\ep}(\zeta)=\prod_{j=1}^m\big(\zeta-\lambda_j(\ep)\big),\quad \lambda_1(\ep)<\lambda_2(\ep)<\cdots <\lambda_m(\ep)
\]
and constants $r\geq 0$, $s>0$ and  $c>0$, $C>0$ independent of $\ep>0$ such that
\begin{eqnarray}
\label{eq:quasi:cond:1}
c\,\ep^{r}\leq \big|p'_{\ep}\big(\lambda_j(\ep)\big)\big|,\\
\label{eq:quasi:cond:2}
\big|q_{\ep}\big(\lambda_j(\ep)\big)\big|\leq C\ep^s\,\big|p'_{\ep}\big(\lambda_j(\ep)\big)\big|
\end{eqnarray}
for $j=1,\ldots,m$ where $p'_{\ep}(\zeta)=\dif p_{\ep}/\dif \zeta$ and $q_{\ep}(\zeta)=p(\zeta)-p_{\ep}(\zeta)$.
\begin{prop}
\label{pro:quasi:gene} Assume that $p_{\ep}(\zeta)$ verifies \eqref{eq:quasi:cond:1}, \eqref{eq:quasi:cond:2} and let $H_{\ep}$ be the B\'ezout matrices of $p_{\ep}$ and $p'_{\ep}$. Then there exists $C>0$ independent of $\ep$ such that 
\begin{equation}
\label{eq:quasi:def}
\begin{split}
 &\ep^{2r}|z|^2\leq C\,(H_{\ep}z,z),\quad \forall z\in \C^m,\\
& \big|\big((H_{\ep}A-A^*H_{\ep})z,w\big)\big|\leq C\,\ep^s\,(H_{\ep}z,z)^{1/2}(H_{\ep}w,w)^{1/2},\quad \forall z,w\in \C^m.
\end{split}
\end{equation}
\end{prop}
\begin{proof} Denote by $R_{\ep}$ and $G_{\ep}$ which are defined by replacing $\lambda_j$ by $\lambda_j(\ep)$. Let $H_{\ep}$ be the B\'ezout matrix of $p_{\ep}$ and $p'_{\ep}=\dif p_{\ep}/\dif \zeta$. Noting that $\alpha_k=1$ in \eqref{eq:hokan} one has $H_{\ep}={^t}G_{\ep}G_{\ep}$. Since 
$\lambda_j(\ep)$ are bounded uniformly in $\ep>0$, which follows from \eqref{eq:quasi:cond:2}, it is clear that there is $C_1>0$ independent of $\ep$ such that
\begin{equation}
\label{eq:RtoC}
\big|R_{\ep}z|\leq C_1|z|,\qquad |G_{\ep}z|\leq C_1|z|.
\end{equation}
From \eqref{eq:RtoC} one has $|z|\leq C|R^{-1}_{\ep}z|$ then it follows from \eqref{eq:Cform} and \eqref{eq:quasi:cond:1} that there is $c>0$ such that $|G_{\ep}z|\geq c\,\ep^{r}|z|$. This implies
\[
c^2\,\ep^{2r}|z|^2\leq (H_{\ep}z,z)
\]
for $H_{\ep}={^tG}_{\ep}G_{\ep}$. Denoting by $A$ and $A_{\ep}$ the Sylvester matrices associated with $p$ and $p_{\ep}$ respectively, one has
\begin{equation}
\label{eq:Atoep}
A=A_{\ep}+Q_{\ep}
\end{equation}
where $Q_{\ep}$ is $m\times m$ matrix whose first $m-1$ rows are zero and of which last row consists of the coefficients of $q_{\ep}$, that is $-\big(b_m(\ep),b_{m-1}(\ep),\ldots,b_1(\ep)\big)$ where $q_{\ep}(\zeta)=\sum_{j=0}^mb_{m-j}(\ep)\zeta^j$. Since  $H_{\ep}A_{\ep}$ is symmetric and hence
\[
H_{\ep}A-A^*H_{\ep}=H_{\ep}Q_{\ep}-Q^*_{\ep}H_{\ep}
\]
by \eqref{eq:Atoep}. It is easy to see from the definition that all entries of $Q_{\ep}R_{\ep}$ are zero except for the last row and the last row is $-\big(q_{\ep}(\lambda_1(\ep)),\ldots, q_{\ep}(\lambda_m(\ep))\big)$. Recall that \eqref{eq:Cform} gives
\[
R_{\ep}^{-1}={\rm diag}\,\big(\pm p'_{\ep}(\lambda_1(\ep))^{-1},
\ldots, \pm p'_{\ep}(\lambda_m(\ep))^{-1}\big)G_{\ep}
\]
then one can write $
Q_{\ep}=\big(Q_{\ep}R_{\ep}\big)R_{\ep}^{-1}=S_{\ep}G_{\ep}$  
where the last row of  $S_{\ep}$ is 
\[
\mp\,\big( q_{\ep}(\lambda_1(\ep))p'_{\ep}(\lambda_1(\ep))^{-1},\ldots,q_{\ep}(\lambda_m(\ep))p'_{\ep}(\lambda_m(\ep))^{-1}\big)
\]
and hence $|S_{\ep}z|\leq C\ep^s\,|z|$ for $z\in \C^m$ thanks to \eqref{eq:quasi:cond:2}. Therefore from \eqref{eq:RtoC} one concludes
\begin{align*}
\big|\big(H_{\ep}Q_{\ep}z,w\big)\big|=\big|\big(S_{\ep}G_{\ep}z, H_{\ep}w\big)\big|\leq \ep^s\,C\big|G_{\ep}z\big|\big|H_{\ep}w\big|\\
\leq \ep^s\,CC_1\big|G_{\ep}z\big|\big|G_{\ep}w\big|=\ep^s\,C_2\big(H_{\ep}z,z\big)^{1/2}\big(H_{\ep}w,w\big)^{1/2}.
\end{align*}
Since the estimate for $\big|\big(Q^*_{\ep}H_{\ep}z,w\big)\big|=\big|(H_{\ep}z,Q_{\ep}w)\big|$ is  same as above one completes the proof.
\end{proof}
The following corollary is found in \cite{Nbook}.
\begin{cor}
\label{cor:quasi:sym}Assume that the multiplicity of any root of $p(\zeta)=0$ does not exceed $\rho$ then   the B\'ezout matrices $H_{\ep}$ of Nuij approximation $p_{\ep}$ and $p'_{\ep}$ is quasi-symmetrizers, that is $H_{\ep}$ verifies \eqref{eq:quasi:def} with $r=\rho-1$ and $s=1$.
\end{cor}
\begin{proof} It suffices to check \eqref{eq:quasi:cond:1} and \eqref{eq:quasi:cond:2} with $r=\rho-1$ and $s=1$. 
Since the multiplicity of the roots are at most $\rho$ it is clear from Lemma \ref{lem:pq:kinji} that
\[
|p'_{\ep}(\lambda_j(\ep))|=\prod_{k=1, k\neq j}^m\big|\lambda_j(\ep)-\lambda_k(\ep)\big|\geq c\, \ep^{\rho-1},\quad j=1,\ldots,m.
\]
Note that  one can invert $ (1+\ep  d/d\zeta)^{m-1}p(\zeta)=p_{\ep}(\zeta)$ such that
\[
p(\zeta)=p_{\ep}(\zeta)+c_1\ep\,p_{\ep}^{(1)}(\zeta)+\cdots+c_m\ep^m\,p_{\ep}^{(m)}(\zeta),\quad p^{(\ell)}_{\ep}(\zeta)=d^{\ell}p_{\ep}(\zeta)/d\zeta^{\ell}.
\]
and hence $q_{\ep}=\sum_{\ell=1}^mc_j\,\ep^{\ell}\,p^{(\ell)}_{\ep}$. 
Since
\[
p^{(\ell)}_{\ep}(\lambda_j(\ep))=\sum_{1\leq k_1<\cdots<k_{m-\ell}\leq m, k_i\neq j}\prod \big(\lambda_j(\ep)-\lambda_{k_i}(\ep)\big)
\]
it is clear from \eqref{eq:nuij:sa} that
\[
\big|p^{(\ell)}_{\ep}(\lambda_j(\ep))\big|/\big|p'_{\ep}(\lambda_j(\ep))\big|\leq C\,\ep^{-(\ell-1)},\quad \ell=1,\ldots,m
\]
which proves \eqref{eq:quasi:cond:2} with $s=1$.
\end{proof}
%

\section{Rrmarks}

Assume that $p(\zeta)$ is a strictly hyperbolic polynomial. From \eqref{eq:AR=RD} it follows that $R^{-1}AR$ is diagonal and hence symmetric which shows $(R\,^t\!R)A=A(R\,^t\!R)$. Then with $S=R\,^t\!R$ one sees that $AS$ is symmetric. Since $S$ is symmetric 
\[
S^{-1}A=S^{-1}(AS)S^{-1}
\]
is also symmetric. On the other hand, denoting $S=(s_{ij})$ it is clear that $s_{ij}=\sum_{k=1}^m\lambda_k^{i+j}$ which is a symmetric polynomial in $(\lambda_1,\ldots,\lambda_m)$ and hence a polynomial in $(a_1,\ldots,a_m)$. Denote $B=({\rm det}\,S)S^{-1}$ then  $BA$ is symmetric and $B$ is positive definite because ${\rm det}\,S=({\rm det}\,R)^2=\Delta^2>0$. Since $B$ is the cofactor matrix of $S$ then every entry of $B$ is also a polynomial in $(a_1,\ldots,a_m)$. This $B$ is the symmetrizer which was used to derive energy estimates for strictly hyperbolic equations in \cite[Chapter V]{Leray}. From \eqref{eq:Cform} one can write $H={^t}\!R^{-1}{\rm diag}\,\big((p'(\lambda_1))^2,\ldots,(p'(\lambda_m))^2\big)R^{-1}$.  Then it is clear that
\[
HR\,{\rm diag}\,\big((p'(\lambda_1))^{-2},\ldots,(p'(\lambda_m))^{-2}\big)R^{-1}={^t}\!R^{-1}R^{-1}=\Delta^{-2}\,B.
\]
In particular, $B=H$ if $m=2$ and ${\rm det}\,B=\Delta^{2(m-1)}$.

%
%

Symmetrizations by B\'ezoutiant or quasi-symmetrizers are applied to several problems by several authors, see for example, \cite{Ja1}, \cite{Ja2}, \cite{JaTa}, \cite{DaSpa}, \cite{KiSpa}, \cite{SpaTa2}, \cite{NP}. In particular, interesting results for the Cauchy problem for differential operators with time dependent coefficients are obtained in \cite{JaTa} based on detailed study on $H$, while quasi-symmetrizers $H_{\ep}$ is applied to study  propagation of the analyticity for a class of semilinear weakly hyperbolic systems in \cite{DaSpa}.

\section{Proof of Lemmas}

First we give a proof of Lemma \ref{lem:fourone}. 

\smallskip
\noindent
Proof of Lemma \ref{lem:fourone}: 
 If $s=1$ then the assertion is clear. Let $s\geq 2$ and denote
$$
a(\zeta)=\prod_{j=1}^s(\zeta-\lambda_{(j)}),\qquad b(\zeta)=\prod_{j=1}^{s-1}(\zeta-\mu_j)
$$
so that $
b(\zeta)\big\{\prod_{j=1}^s(\zeta-\lambda_{(j)})^{r_j-1}\big\}=q(\zeta)$.  
Set
\begin{equation}
a_k(\zeta)=\prod_{j\neq k}^s(\zeta-\lambda_{(j)}),\quad \al_k=\frac{b(\lambda_{(k)})}{
a_k(\lambda_{(k)})}>0.
\label{eqfourone}
\end{equation}
Writing $b(\zeta)=\sum_{k=1}^s\alpha_k\,a_k(\zeta)$, the same argument as before gives
$$
a(\zeta)b({\bar\zeta})-a({\bar\zeta})b(\zeta)=(\zeta-{\bar\zeta})\sum_{k=1}^s\al_k\,a_k(\zeta)a_k({\bar\zeta}).
$$
Now we have
\begin{eqnarray*}
\frac{p(\zeta)q({\bar\zeta})-p({\bar\zeta})q(\zeta)}{\zeta-{\bar\zeta}}=\frac{\Big|\prod_{j=1}^s(
\zeta-\lambda_{(j)})^{r_j-1}\Big|^2\big(a(\zeta)b({\bar\zeta})-a({\bar\zeta})b(\zeta)\big)}{\zeta-{\bar\zeta}}\\
=\sum_{k=1}^s\al_k\Big|\prod_{j=1}^s(
\zeta-\lambda_{(j)})^{r_j-1}\Big|^2|a_k(\zeta)|^2=
\sum_{k=1}^s\al_k\Big|\prod_{j=1}^s(
\zeta-\lambda_{(j)})^{r_j-\delta_{kj}}\Big|^2
\end{eqnarray*}
where $\delta_{kj}$ is the Kronecker's delta. This proves that
\begin{equation}
h_{p,q}(\zeta,{\bar\zeta})=\sum_{k=1}^s\al_k\,\phi_k(\zeta)\phi_k(\bar \zeta),\quad \phi_k(\zeta)=\prod_{j=1}^s(\zeta-\lambda_{(j)})^{r_j-\delta_{kj}}\;.
\label{eqfourtwo}
\end{equation}
Since 
$$
\sum_{k=1}^m\big|p_k(\zeta)\big|^2=\sum_{k=1}^m\prod_{j\neq  k}^m|\zeta-\lambda_j|^2=
\sum_{k=1}^sr_k\,\phi_k(\zeta)\phi_k({\bar\zeta)}
$$
we get the desired inequality
\[
{\hat h}(z,{\bar z})=\sum_{k=1}^s\alpha_k\,\phi_k(z)\phi_k({\bar z)}\geq c \sum_{k=1}^sr_k\,\phi_k(z)\phi_k({\bar z)}=\sum_{j=1}^m|{\hat p}_k(z)|^2
\]
with $c=\min{\alpha_k/r_k}$.
We turn to the proof of the converse. Note that \eqref{eq:bunri} implies $h_{p,q}(\zeta,{\bar\zeta})\geq c\,\sum_{j=1}^m|p_k(\zeta)|^2$. Since
\begin{equation}
\frac{\dif p}{\dif \zeta}(\zeta)q(\zeta)-p(\zeta)\frac{\dif q}{\dif \zeta}(\zeta)
=h_{p,q}(\zeta,\zeta),\quad \zeta\in\R
\label{eqfourthree:bis}
\end{equation}
it is clear from the assumption that the zeros of $q$ other than $\{\lambda_{(j)}\}$ 
are simple. It is also clear from \eqref{eqfourthree:bis} that the coefficient of $\zeta^{m-1}$ 
in $q$ is positive. We examine that $q$ has no zero in $(-\infty,\lambda_{(1)})$. If 
there were, we denote  the minimal one by $\mu$. Then we see that
$$
\frac{\dif q}{\dif \zeta}(\mu)>0\;(<0)
$$
if $m$ is even (odd). On the other hand $p(\mu)$ has the sign $(-1)^m$ it follows 
that
$$
-p(\mu)\frac{\dif q}{\dif \zeta}(\mu)<0
$$
and hence $h_{p,q}(\mu,\mu)<0$, contradicting the assumption. We then examine that 
$q$ has no zero in $\zeta>\lambda_{(s)}$. This can be checked by a similar way. We next 
show that $q$ has at most one zero in each $(\lambda_{(k)},\lambda_{(k+1)})$. If not there 
were two successive simple zeros $\mu_i\in (\lambda_{(k)},\lambda_{(k+1)})$, $i=1,2$ and hence 
$p(\zeta)\cdot\dif q(\zeta)/\dif \zeta$ has different signs at $\mu_1$ and $\mu_2$ and 
hence a contradiction. Thus we can conclude that either $q(\zeta)$ separates $p(\zeta)$ 
or some $\lambda_{(j)}$ is a zero of $q(\zeta)$ of order greater than $r_j-1$. Suppose that 
this is the case. Then one can write
$$
q(\zeta)=(\zeta-\lambda_{(j)})^lr(\zeta),\quad l\geq r_j.
$$
Taking $\zeta=\lambda_{(j)}+\xi$ we see that the right-hand side of \eqref{eqfourthree:bis} is 
$O(|\xi|^{l+r_j-1})$. On the other hand it is clear that
$$
|h_{p,q}(\zeta,\zeta)|\geq c\,|\xi|^{2(r_j-1)}
$$
with some $c>0$. This contradicts the assumption.
\qed

\medskip
Noting that $
p(\zeta)p'({\bar\zeta})=\dif(p(\zeta)p({\bar\zeta}))/\dif {\bar\zeta}$ and $p'(\zeta)p({\bar\zeta})=
\dif (p(\zeta)p({\bar\zeta}))/{\dif \zeta}
$
one has
$$
h_{p,p'}(\zeta,{\bar\zeta})=\Big(\frac{\dif}{\dif\bar\zeta}
\prod_{j=1}^m (\zeta-\lambda_j)({\bar\zeta}-\lambda_j)-\frac{\dif}{\dif\zeta}
\prod_{j=1}^m (\zeta-\lambda_j)({\bar\zeta}-\lambda_j)\Big)/(\zeta-{\bar\zeta}).
$$
Since $
(\dif/\dif\bar\zeta-\dif/\dif\zeta)(\zeta-\lambda_j)({\bar\zeta}-\lambda_j)
=\zeta-{\bar\zeta}$
one has
\begin{equation}
\label{eq:q=p'}
h_{p,p'}(\zeta,{\bar\zeta})=\sum_{k=1}^m\prod_{j\neq k}^m|\zeta-\lambda_j|^2=\sum_{k=1}^m|p_k(\zeta)|^2.
\end{equation}
Recalling  $
p'(\zeta)=\sum_{k=1}^sr_k
\phi_k(\zeta)$
one  obtains
$$
|\hat p'(z)|^2\leq \Big(\sum_{k=1}^sr_k^2\al_k^{-1}\Big)
\Big(\sum_{k=1}^s\al_k|\hat \phi_k(z)|^2\Big)=
\Big(\sum_{k=1}^sr_k^2\al_k^{-1}\Big)\hat h_{p,p'}(z,\bar z)
$$
which together with \eqref{eq:q=p'} proves Lemma \ref{lem:fourfour}.

\smallskip

Next we give a proof of Lemma \ref{lem:pq:kinji}.

\noindent
Proof of Lemma \ref{lem:pq:kinji}: 
Let $p$ be a monic hyperbolic polynomial and $q$ be a hyperbolic polynomial which separates $p$. To prove the first assertion it suffices to prove that
writing
\[
\big(1+\ep \dif/\dif \zeta\big)q=c\prod_{j=1}^{m-1}(\zeta-\mu_j(\ep)),\quad \big(1+\ep \dif/\dif \zeta\big)p=\prod_{j=1}^m(\zeta-\lambda_j(\ep))
\]
where $\lambda_1(\ep)\leq \cdots \leq \lambda_{m}(\ep)$ and $ \mu_1(\ep)\leq\cdots \leq \mu_{m-1}(\ep)$ one has 
\begin{equation}
\label{eq:bunri:1}
\lambda_1(\ep)\leq \mu_1(\ep)\leq \lambda_2(\ep)\leq \cdots\leq \mu_{m-1}(\ep)\leq \lambda_m(\ep)
\end{equation}
and  the multiplicity of multiple  roots decreases by one by this procedure for $\ep>0$. Let $p(\zeta)=\prod_{j=1}^m(\zeta-\lambda_j)$, $\lambda_1\leq \lambda_2\leq\cdots \leq \lambda_m$ and $q(\zeta)=c\prod_{j=1}^{m-1}(\zeta-\mu_j)$, $\mu_1\leq\mu_2\cdots\leq \mu_{m-1}$ and assume that
\[
\lambda_1\leq \mu_1\leq \lambda_2\leq \cdots \leq \mu_{m-1}\leq \lambda_m.
\]
Denote 
\begin{equation}
\label{eq:FtoG}
\begin{split}
F(\zeta)=\frac{(1+\ep\,d/d\zeta)p(\zeta)}{p(\zeta)}=1+\ep\,\sum_{j=1}^m\frac{1}{\zeta-\lambda_j},\\
G(\zeta)=\frac{(1+\ep\,d/d\zeta)q(\zeta)}{q(\zeta)}=1+\ep\,\sum_{j=1}^{m-1}\frac{1}{\zeta-\mu_j}.
\end{split}
\end{equation}
Noting that $dF/d\zeta$ is strictly negative on each interval not including $\lambda_k$ and $lim_{|\zeta|\to\infty}F=1$ there is a simple root of $(1+\ep\,\dif/\dif \zeta)p=0$ to the left of $\lambda_k$. This proves that the multiplicity of multiple  roots of $(1+\ep\,\dif/\dif \zeta)p=0$ decreases by one from that of $p$. The same for $q$.

Assume $(\lambda_k,\mu_k)\neq \emptyset$. Let $\zeta\in (\lambda_k,\mu_k)$ then $\zeta-\mu_k<0<\zeta-\lambda_{k}$ and hence
\[
\zeta-\lambda_1\geq \zeta-\mu_1\geq \cdots\geq \zeta-\lambda_k>0>\zeta-\mu_k\geq \zeta-\lambda_{k+1}\geq \cdots\geq \zeta-\lambda_m.
\]
This implies
\[
\sum_{j=k}^{m-1}\frac{1}{\zeta-\mu_j}\leq 
\sum_{j=k+1}^{m}\frac{1}{\zeta-\lambda_j},\qquad \sum_{j=1}^{k-1}\frac{1}{\zeta-\mu_j}\leq \sum_{j=2}^{k}\frac{1}{\zeta-\lambda_j}
\]
and hence
\begin{equation}
\label{eq:FlargeG}
G(\zeta)<G(\zeta)+\frac{1}{\zeta-\lambda_1}\leq F(\zeta),\quad \zeta\in (\lambda_k,\mu_k).
\end{equation}
Note that $\lambda_{k}\leq \lambda_{k+1}(\ep)\leq \lambda_{k+1}$ and $\mu_{k-1}\leq \mu_k(\ep)< \mu_k$. If $\lambda_{k+1}(\ep)\geq \mu_k$ then $\mu_k(\ep)< \lambda_{k+1}(\ep)$ is obvious. If $\lambda_k\leq \lambda_{k+1}(\ep)< \mu_k$ then thanks to \eqref{eq:FlargeG} one concludes $\mu_k(\ep)<\lambda_{k+1}(\ep)$. Therefore one has
\[
\mu_k(\ep)<\lambda_{k+1}(\ep).
\]
Next assume $(\mu_k,\lambda_{k+1})\neq \emptyset$. Let $\zeta\in (\mu_k,\lambda_{k+1})$ so that $\zeta-\lambda_{k+1}<0<\zeta-\mu_k$ and hence
\[
\zeta-\lambda_1\geq \zeta-\mu_1\geq \cdots\geq \zeta-\mu_k>0>\zeta-\lambda_{k+1}\geq \zeta-\mu_{k+1}\geq \cdots\geq \zeta-\lambda_m.
\]
This shows that
\[
\sum_{j=1}^{k}\frac{1}{\zeta-\lambda_j}\leq 
\sum_{j=1}^{k}\frac{1}{\zeta-\mu_j},\qquad \sum_{j=k+1}^{m-1}\frac{1}{\zeta-\lambda_j}\leq \sum_{j=k+1}^{m-1}\frac{1}{\zeta-\mu_j}
\]
and hence
\begin{equation}
\label{eq:GlargeF}
G(\zeta)>G(\zeta)+\frac{1}{\zeta-\lambda_m}\geq F(\zeta),\quad \zeta\in (\mu_k,\lambda_{k+1}).
\end{equation}
Note that $\mu_k\leq \mu_{k+1}(\ep)\leq \mu_{k+1}$ and $\lambda_{k+1}(\ep)<\lambda_{k+1}$. If $\mu_{k+1}(\ep)\geq \lambda_{k+1}$ then $\lambda_{k+1}(\ep)<\mu_{k+1}(\ep)$ is obvious. If $\mu_{k+1}(\ep)<\lambda_{k+1}$ then \eqref{eq:GlargeF} shows the same conclusion. Thus one has
\[
\lambda_{k+1}(\ep)<\mu_{k+1}(\ep).
\]
Repeating the same arguments  in $(-\infty,\lambda_1)$ one obtains $\lambda_1(\ep)<\mu_1(\ep)$. Then one concludes \eqref{eq:bunri:1} and hence assertion.
 
Turn to the second assertion which is found in \cite{Waka}. Write
\begin{equation}
\label{eq:ltol+1}
h_{\ell}(\zeta,\ep)=\frac{(1+\ep\,d/d\zeta)^{\ell}p(\zeta)}{(1+\ep\,d/d\zeta)^{\ell-1}p(\zeta)}=1+\ep\,\sum_{j=1}^m\frac{1}{\zeta-\lambda^{\ell}_j(\ep)},\;\;\ell=1,\ldots,m
\end{equation}
where $(1+\ep\,d/d\zeta)^{\ell-1}p(\zeta)=\prod_{j=1}^m\big(\zeta-\lambda_j^{\ell}(\ep)\big)$, $\lambda_1^{\ell}(\ep)\leq \lambda_2^{\ell}(\ep)\leq \cdots\leq \lambda_m^{\ell}(\ep)$.
Since $\lambda_1^{\ell}(\ep)$, $l\geq 2$, $\ep>0$ are simple roots it follows from \eqref{eq:ltol+1} that
\begin{align*}
&\lambda_1^{\ell+1}(\ep)\leq \lambda_1^{\ell}(\ep)\leq \lambda_2^{\ell+1}(\ep)\leq \cdots\leq \lambda_m^{\ell+1}(\ep)\leq \lambda_m^{\ell}(\ep),\\
&\lambda_1^{\ell}(\ep)<\lambda_2^{\ell}(\ep)<\cdots<\lambda_{\ell-1}^{\ell}(\ep)<\lambda_{\ell}^{\ell}(\ep)\leq \cdots\leq \lambda_m^{\ell}(\ep)
\end{align*}
where $\lambda_k^{\ell}(\ep)$, $1\leq k\leq \ell-1$ are simple roots. Assume that there is $c_{\ell}>0$ such that
\begin{equation}
\label{eq:kinou}
\lambda_k^{\ell}(\ep)-\lambda_{k-1}^{\ell}(\ep)\geq c_{\ell}\,\ep,\quad k=2,\ldots,\ell.
\end{equation}
It is easy to see that \eqref{eq:kinou} holds for $\ell=2$ with $c_2=1$. It follows from \eqref{eq:ltol+1} that
\[
h_{\ell}(\lambda_k^{\ell}(\ep)-\delta\ep,\ep))\leq 1+\frac{\ep(k-1)}{\lambda_k^{\ell}(\ep)-\delta\ep-\lambda_{k-1}^{\ell}(\ep)}-\frac{1}{\delta}\leq 1+\frac{k-1}{c_{\ell}-\delta}-\frac{1}{\delta}
\]
for $2\leq k\leq \ell$, $0<\delta\leq c_{\ell}$. 
Therefore choosing $\delta=\big(k+c_{\ell}-\sqrt{(k+c_{\ell})^2-4c_{\ell}}\big)/2$ one has $h_{\ell}(\lambda_k^{\ell}(\ep)-\delta\ep,\ep)\leq 0$ and hence $\lambda_k^{\ell+1}(\ep)\leq \lambda_k^{\ell}(\ep)-\delta\ep$. Then taking
\[
c_{\ell+1}=\min_{2\leq k\leq \ell}{\big(k+c_{\ell}-\sqrt{(k+c_{\ell})^2-4c_{\ell}}\big)/2}>0
\]
one has $\lambda_{k+1}^{\ell+1}(\ep)-\lambda_k^{\ell+1}(\ep)=\lambda_{k+1}^{\ell+1}(\ep)-\lambda_k^{\ell}(\ep)+\lambda_k^{\ell}(\ep)-\lambda_k^{\ell+1}(\ep)\geq \lambda_k^{\ell}(\ep)-\lambda_k^{\ell+1}(\ep)\geq c_{\ell+1}\,\ep$ for $k=1,\ldots,\ell$. Thus \eqref{eq:ltol+1} holds for $\ell=m$ by induction.
 \qed

%
%

\end{document}